\documentclass[11pt]{amsart}
\usepackage{amsfonts}
\usepackage{mathrsfs}
\usepackage{amsmath}
\usepackage{amsthm}
\usepackage{tikz}

\def\av{\ensuremath{\text{Av}_{n}}}
\def\sqbox#1{\fbox{\ensuremath{#1}}}
\newtheorem{theorem}{Theorem}[section]
\newtheorem{lemma}[theorem]{Lemma}
\newtheorem{corollary}[theorem]{Corollary}
\def\inj{\hookrightarrow}
\theoremstyle{definition}
\newtheorem*{example}{Example}

\begin{document}

\title{Permutations Almost Avoiding Monotone Distant Patterns}

\author{Nicholas Van Nimwegen}
\date{}
\begin{abstract}
In \cite{bona}, Bóna and Pantone studied permutations that avoided all but one pattern of length $k$ that began with a length $k-1$ increasing subsequence. We draw the connection between that idea and distant patterns, first discussed heavily in \cite{dimitrov}, and study similar permutation classes, where the index not part of the increasing subsequence can vary. We find a large class of Wilf-Equivalences between $k+1$ classes of $k$ patterns of length $k+1$, and outline several classes of unbalanced Wilf-Equivalences related to the first class. Using this, we are also find new bounds on the exponential growth rate on all monotone distant patterns with a single gap constraint.
\end{abstract}
\maketitle

\section{Introduction}
The study of permutation patterns has been a major area of research in combinatorics for the past 50 years. In that time, several varieties of permutation patterns have been introduced. Of particular interest to us is the concept of distant patterns, or patterns with constraints placed on gap sizes between entries. These were first systematically discussed by Dimitrov in \cite{dimitrov}, who considered all distant patterns on 2 and 3 letters, and mentioned several unproved conjectures for longer distant patterns.

A classical distant pattern can equivalently be thought of as a collection of classical patterns. Specifically, a classical distant pattern on $k$ letters with one constraint of size 1 can be thought of as a collection of $k+1$ classical patterns, all on $k+1$ letters. We introduce in this work the concept of \textit{almost distant patterns}, where one of the $k+1$ classical patterns is removed.

We consider strictly monotone patterns, and in doing so find a surprising Wilf-Equivalence between a large collection of almost distant patterns. We use this method to place an upper bound on the growth rate of all monotone distant patterns, which we believe to be better than any existing bounds. In addition, we find an infinite class of unbalanced Wilf-Equivalences relating to almost distant patterns, which are always of interest given the relative rarity of such equivalences. We conclude with several open questions.

It should be mentioned that this is not the first paper to consider almost distant patterns, although it is the first to give them that name. In \cite{bona}, Bóna and Pantone considered collections of permutations that are exactly almost distant patterns with constrained gap following the pattern, and their results on growth rates are the way that we provide bounds on distant patterns in general.

\subsection{Definitions}
All permutations in this paper will be presented in one-line notation. Let $q = q_{1}\dots q_{k} \in \mathfrak{S}_{k}$, $p = p_{1}\dots p_{n} \in \mathfrak{S}_{n}$. We say $p$ \textit{avoids} $q$ if there is no subsequence $p_{i_{1}}\dots p_{i_{k}}$ with $i_{1} < i_{2} < \dots < i_{k}$ such that $q_{a} < q_{b} \iff p_{i_{a}} < p_{i_{b}}$. Otherwise, we say $p$ \textit{contains} $q$. If $p$ avoids $q$, we say $p \in \av(q)$.

We will use $\square$ to denote a gap between letters. Thus avoiding the pattern $q_{1}\dots q_{j-1}\square q_{j}\dots q_{k}$ is equivalent to avoiding occurrences $p_{i_{1}}\dots p_{i_{k}}$ of $q$ where $q_{j-1}$ and $q_{j}$ are not adjacent. We call a pattern with a $\square$ a \textit{distant pattern}, where the length of the distant pattern is the number of letters excluding the $\square$. So our definition would be a distant pattern of length $k$.

Note the general definition of distant patterns allows for multiple squares, and also allows for $\square^{r}$ to denote gaps of size at least $r$. Our results relate specifically to distant patterns with only a single square forbidding any gap, so our definition only considers that case.

It was shown in \cite{dimitrov} that avoiding a distant pattern as defined above  is equivalent to avoiding $k+1$ patterns of length $k+1$. In particular, if $q = q_{1}\dots q_{j-1}\square q_{j}\dots q_{k}$ is a distant pattern of length $k$, and $q' = q_{1}\dots q_{k}$ is the classical pattern obtained by ignoring the $\square$, the $k+1$ patterns in question are those where the $j^{\text{th}}$ entry ranges from $1$ to $k+1$, and the remaining entries are an occurrence of $q'$. For example, $12\square 3 = \{2314,1324,1234,1243\}$.

Our definition of almost distant patterns comes directly from this equivalence. We define $q = q_{1}\dots q_{j-1}\sqbox{i}q_{j}\dots q_{k}$ as the collection of $k+1$ patterns equivalent to $q_{1}\dots q_{j-1}\square q_{j}\dots q_{k}$, excluding the pattern with $i$ in the $j^{\text{th}}$ position. For example, $12\sqbox{2}3 = \{2314,1234,1243\}$. We call $q$ an \textit{almost distant pattern} of length $k$.

While the idea of almost distant patterns can have any classical pattern underlying them, this work relates to only almost distant monotone patterns. For that reason, we introduce notation to simplify presentation. We write $M_{k,j,i} = 12\dots (j-1)\sqbox{i}j(j+1)\dots (k-1)k$. In this manner, the permutation classes discussed in \cite{bona} as $A_{k,i}$ can be written as $M_{k-1,k,i}$.

We will also be concerned with the growth rates of almost distant patterns. Given a set of permutations $Q$, define the \textit{exponential growth rate} of the sequence $\av(Q)$ as $L(Q) = \lim_{n \rightarrow \infty} \av(Q)^{1/n}$. In \cite{bona}, Bóna and Pantone showed $L(M_{k,k+1,i}) = (k-1)^{2}$ for $2 \leq i \leq k$, and $L(M_{k,k+1,k+1}) = (k-1)^{2}+1$.

\section{Wilf-Equivalence for $i=j$}

We first introduce a small lemma that will be instrumental in our proof:

\begin{lemma}
	Let $k \in \mathbb{Z}^{+}$. Then $\av(M_{k,j,i}) = \av(M_{k,k+2-j,k+2-i})$ for all $n \in \mathbb{Z}^{+}$, $1 \leq i,j \leq k+1$.
\end{lemma}

\begin{proof}
Recall if $q = q_{1}\dots q_{t} \in \mathfrak{S}_{t}$, the reverse complement of $q$, denoted $q^{rc}$, is $(t+1-q_{t})(t+1-q_{t-1})\dots (t+1-q_{2})(t+1-q_{1})$. It is well known that if $P,Q$ are two sets of permutations such that $P^{rc} = Q$, then $\av(P) = \av(Q)$ for all $\mathbb{Z}^{+}$. The proof of the lemma then follows from noticing $M_{k,j,i} = M_{k,k+2-j,k+2-i}^{rc}$.
\end{proof}

With that out of the way, we are ready to present our main result:

\begin{theorem}
	Let $k \in \mathbb{Z}^{+}$, then the almost distant patterns $M_{k,i+1,i+1}$ are Wilf-Equivalent for all $0 \leq i \leq k$.
\end{theorem}

\begin{proof}
First note that by lemma, $M_{k,1,1} = M_{k,k+1,k+1}$. Our proof is as follows: we will construct an injection $F$ from $M_{k,i+1,i+1} \inj M_{k,i+2,i+2}$ for $0 \leq i \leq k-1$. We will then have a string of injections $M_{k,1,1} \inj M_{k,2,2} \inj \dots \inj M_{k,k+1,k+1} \inj M_{k,1,1}$, where the final map is the reverse complement map. We can then compose these maps to get an injection $M_{k,i+1,i+1} \inj M_{k,i+1,i+1}$ for $0 \leq i \leq k$, and thus each map must be a bijection.

Let $p \in M_{k,i+1,i+1}$. We define 3 sets, $A,B,C$. The set $B$ will be all elements in $p$ that can act as an $i+1$ in a $12\dots k$ pattern. The set $A$ will be all elements in $p$ that can act as an $i$ in a $12\dots k$ pattern, that are not already in $B$. Finally, the set $C$ will be the set of elements of $p$ that can act as an $i+2$ in a $12\dots k$ pattern, that are not already elements of $B$.

\begin{example}
Consider the permutation $p = 82456173 \in Av_{8}(M_{4,3,3})$. Here $A = \{4\}, B = \{5\}, C = \{6\}$.
\end{example}

This definition of sets does raise the question of several edge cases. In particular, it raises the question of how we define $A$ for $i = 0$, and $C$ for $i = k-1$. For $i=0$, we set $A$ to the start of the permutation, and for $i=k-1$ we set $C$ to the end of the permutation. The way that this affects the following facts and the map itself will be addressed when relevant.

Before we define the map, we should establish some facts about our set $B$. First, each element $b \in B$ cannot have multiple $a \in A$ such that $a,b$ can act as an $i,i+1$ together in a $12\dots k$ pattern. By way of contradiction assume there are two such elements, $a,a'$, with $a < a'$ without loss of generality. Then either $a$ comes before $a'$, in which case $a'$ could be an $i+1$ (and thus be in $B$), or $a$ comes after $a'$. In that case, the $12 \dots k$ pattern with $a',b$ together with $a$ would be one of our forbidden patterns.

For the edge case $i=0$, we obviously still get that there are not multiple elements of $A$ associated to a given element of $B$.

Next let $b \in B$, and $a \in A$ the unique element of $A$ that acts with it in a $12\dots k$ pattern. Then the subsequence from $a$ to $b$ must be an increasing pattern by our almost distant pattern condition. Further, every element in that subsequence (except $a$ itself) must be an element of $B$, since we can take the $12\dots k$ pattern containing $a,b$, and replace $b$ with any element between $a$ and $b$, since by our forbidden patterns these elements must also have value between $a$ and $b$. This means that every element of $B$ is part of a connected increasing subsequence following an element of $A$.

For the edge case $i=0$, this means that all elements of $B$ occur in an increasing subsequence at the start of the permutation.

With these facts established, we can get into our map. For each element $b \in B$, we define $f(b)$ as the farthest right element of $C$ that is larger than $b$. Our map $F$ takes each element of $B$ to directly before $f(b)$. If multiple elements of $B$ have the same image under $f$, they are placed in increasing order. Note this map does not move any elements not in $B$ relative to each other, and does not move any elements of $B$ backwards relative to the rest of the permutation.

For the edge case $i=k-1$, all elements of $B$ move to the end of the permutation and are placed in increasing order.

\begin{example}
Consider the permutation $p = 8\;3\;2\;11\;12\;5\;6\;9\;10\;14\;4\;1\;13\;7 \in Av_{14}(M_{4,3,3})$.  Then $A \{5,11\}, B = \{6,9,10,12\}, C = \{7,13,14\}$. Then $f(6) = 7$, and $f(8) = f(10) = f(12) = 13$. We show $p$ and $F(p)$ in Figure 1 below, with $A,B,C$ colored in red, blue, and green respectively.
\end{example}

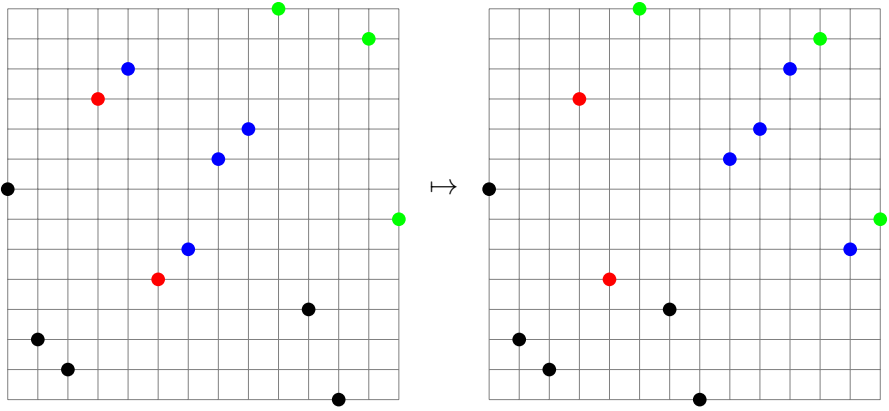
\begin{center}
\begin{tikzpicture}[scale = 0.4]

\draw [help lines] (0,0) grid (13,13);
\node [circle,fill,scale=0.5] at (0,7) {};
\node [circle,fill,scale=0.5] at (1,2) {};
\node [circle,fill,scale=0.5] at (2,1) {};
\node [circle,fill,scale=0.5,red] at (3,10) {};
\node [circle,fill,scale=0.5,blue] at (4,11) {};
\node [circle,fill,scale=0.5,red] at (5,4) {};
\node [circle,fill,scale=0.5,blue] at (6,5) {};
\node [circle,fill,scale=0.5,blue] at (7,8) {};
\node [circle,fill,scale=0.5,blue] at (8,9) {};
\node [circle,fill,scale=0.5,green] at (9,13) {};
\node [circle,fill,scale=0.5] at (10,3) {};
\node [circle,fill,scale=0.5] at (11,0) {};
\node [circle,fill,scale=0.5,green] at (12,12) {};
\node [circle,fill,scale=0.5,green] at (13,6) {};

\node at (14.5,7) {$\mapsto$};

\draw [help lines] (16,0) grid (29,13);
\node [circle,fill,scale=0.5] at (16,7) {};
\node [circle,fill,scale=0.5] at (17,2) {};
\node [circle,fill,scale=0.5] at (18,1) {};
\node [circle,fill,scale=0.5,red] at (19,10) {};
\node [circle,fill,scale=0.5,red] at (20,4) {};
\node [circle,fill,scale=0.5,green] at (21,13) {};
\node [circle,fill,scale=0.5] at (22,3) {};
\node [circle,fill,scale=0.5] at (23,0) {};
\node [circle,fill,scale=0.5,blue] at (24,8) {};
\node [circle,fill,scale=0.5,blue] at (25,9) {};
\node [circle,fill,scale=0.5,blue] at (26,11) {};
\node [circle,fill,scale=0.5,green] at (27,12) {};
\node [circle,fill,scale=0.5,blue] at (28,5) {};
\node [circle,fill,scale=0.5,green] at (29,6) {};

\end{tikzpicture}
\end{center}


To show that this is an injective map into $M_{k,i+2,i+2}$, we need to establish first that $F(p) \in M_{k,i+2,i+2}$, and then show the map is injective.

To show our map does go into $M_{k,i+2,i+2}$, first we prove the following lemma:

\begin{lemma}
Let $p \in \av(M_{k,i+1,i+1})$, and $F$ be as described above. Then for all elements $x$ in $F(p)$, $x$ can act as an $i+1$ in a $12\dots k$ pattern in $F(p)$ if and only if $x \in B$ in $p$.
\end{lemma}

\begin{proof}
First let $x \in B$ in $p$. Then there is a $12\dots k$ pattern in $p$ with $i \in A, (i+1) = x, (i+2) = f(x)$. Then the only element in this pattern that is moved under $F$ is $x$, and $x$ is still before $f(x)$ in $F(p)$. Thus this $12\dots k$ pattern still exists in $F(p)$, so $x$ is an $(i+1)$ in a $12\dots k$ pattern in $F(p)$.

Now let $x$ be an $(i+1)$ in a $12\dots k$ pattern in $F(p)$, then we can find a $12 \dots k$ pattern with $i \in A, (i+1) = x, (i+2) \in C$. But then $x$ must have been an $(i+1)$ in $p$, and thus $x \in B$.
\end{proof}

Since we know we haven't created any new elements that can be $(i+1)$s in $F(p)$, we can show we fulfill our almost distant pattern condition. Define $B'$ as all elements of $F(p)$ that can be an $(i+1)$ in a $12\dots k$ pattern in $F(p)$, and let $b \in B'$. By our map, the only element in $F(p)$ that can act as an $(i+2)$ with $b$ that is not an element of $B'$ is $f(b)$, and every element between $b$ and $f(b)$ has value between $b$ and $f(b)$. Thus, we do not contain any forbidden patterns, and so $F(p) \in M_{k,i+2,i+2}$.

Now we show this map is an injection. Let $F(p) \in M_{k,i+2,i+2}$, we show we can uniquely recover $p$. All elements not in $B'$ are not moved by $F$, and thus their relative position can be recovered. For all elements $b \in B'$, we know that in $p$ there was only one element $a \in A$ smaller than $b$ that was before $b$, and the only elements between $a$ and $b$ are other elements of $B'$. So we can recover the preimage by moving all elements of $B'$ to directly after their associated element of $A$, placing them in increasing order if multiple elements of $B'$ go to the same element of $A$.

For the edge case $i=0$, all elements of $B'$ are moved to the beginning of the permutation in increasing order.

Since we can uniquely recover the preimage of any $F(p)$, we conclude $F$ is an injection.

As discussed above, if we have an injection forward for all $0 \leq i \leq k-1$, we can conclude $M_{k,i+1,i+1}$ is Wilf-Equivalent for all $0 \leq i \leq k$.

\end{proof}

From \cite{bona}, we know the exponential growth rate of $M_{k,k+1,k+1}$ is $(k-1)^{2}+1$, and thus we get the interesting corollary as a result.

\begin{corollary}

Let $k \in \mathbb{Z}^{+}$. Then for all $1 \leq i \leq k-1$, $L(1\dots i\square (i+1)\dots k) \in [(k-1)^{2},(k-1)^{2}+1]$.

\end{corollary}

\begin{proof}

Let $D_{i} = 1\dots i\square (i+1)\dots k$. We can place a lower bound on the growth rate of $\av(D_{i})$ by observing that $p \in \av(12\dots k) \implies p \in \av(D_{i})$. The exponential growth rate of $\av(12\dots k)$ is known to be $(k-1)^{2}$ \cite{regev}, and thus we have a lower bound for $L(D_{i})$.

For our upper bound, we use the result of Theorem 2.2. Since $M_{k,i+1,i+1} \subseteq D_{i}$, we know $L(M_{k,i+1,i+1}) \geq L(D_{i})$, since a permutation avoiding all the patterns of $D_{i}$ then also avoids all the patterns of $M_{k,i+1,i+1}$.

We know from Theorem 2.2 $L(M_{k,i+1,i+1}) = L(M_{k,k+1,k+1})$, and it was proved in \cite{bona} that $L(M_{k,k+1,k+1}) = (k-1)^{2}+1$. Thus, $L(D_{i}) \leq (k+1)^{2}+1$.

\end{proof}

Technically this proof also suffices for $i = 0,k+1$, but it is already known that the exponential growth rate of those classes is $(k-1)^{2}$, so the corollary would not be providing any new information in those cases.

To the best of our knowledge, the previously best known bound on these classes was $k^{2}$, so this is a vast improvement from known results.

\section{Unbalanced Equivalences for $i = j \pm 1$}

In addition to showing that $M_{k,j,j}$ are Wilf-Equivalent for all $1 \leq j \leq k+1$, we can also show two more Wilf-Equivalences to $M_{k,j,j}$.

\begin{theorem}

Let $k \in \mathbb{Z}^{+}$. Then the sets $M_{k,2,2}, M_{k,2,1}, M_{k,k,k+1}$ are Wilf-Equivalent.

\end{theorem}

\begin{proof}

First note $M_{k,2,1}$ is Wilf-Equivalent to $M_{k,k,k+1}$ by reverse-complement, so proving the first equality is sufficient to prove the theorem.

We construct an explicit bijection $G$ between $\av(M_{k,2,2})$ and $\av(M_{k,2,1})$. Let $p \in \av(M_{k,2,2})$. Let $A$ be the elements of $p$ that can can act as a $2$ in a $12\dots k$ pattern in $p$, but not a $1$. For each $a \in A$, let $g(a)$ be the leftmost element of $p$ that can act as a $1$ with $a$ in a $12\dots k$ pattern. Then for all elements $t$ between $g(a)$ and $a$, we know $g(a) < t < a$ by our almost distant pattern condition. Thus, the range $[g(a),a)$ is exactly the set of elements that can act as a $1$ in a $12\dots k$ pattern, with $g(a)$ acting as a $2$. Our map reverses the subsequences $[g(a),a)$ for each $a \in A$.

The range $[g(a),a)$ is increasing for each $a \in A$, and thus by reversing the subsequence, we've ensured that between every pair of elements that can act as a $1,2$ in a $12\dots k$ pattern, there are only elements smaller than the $1$. Thus we map into our desired set. The inverse map is defined in the exact same manner.

To show this is a bijection, we show each map is an injection. Let $G(p) \in \av(M_{k,2,1})$, we show we can uniquely recover $p$. All elements that cannot act as a $1$ in a $12\dots k$ pattern do not move under $G$, so we can recover their values and position in $p$ from just looking at $G(p)$. Similarly, the reversed subsequences can be uniquely determined from $G(p)$ by the decreasing adjacent sequences of elements that can act as a $1$. Thus, we can uniquely recover $p$ from $G(p)$, and so we have an injection. The process is similar for recovering the preimage of a $G^{-1}(p) \in \av(M_{k,2,2})$, with the only difference being the adjacent sequences of elements that can act as $k$ are now increasing instead of decreasing.

Since the two sets have an explicit bijection, we conclude the patterns are Wilf-Equivalent.

\end{proof}

The key detail for this proof is the elements that can act as a 1 with a given 2 form an increasing subsequence of adjacent elements in $M_{k,2,2}$, and a decreasing subsequence of adjacent elements in $M_{k,2,1}$. A natural question to ask would then be if other almost distant patterns exhibiting the same conditions would also be Wilf-Equivalent. For any given $j$, $M_{k,j,j\pm 1}$ forces decreasing subsequences of elements that act as a $j$ (for $i=j-1$) or $j+1$ (for $i=j+1$).

The answer to this question is unfortunately no. If $G'$ is the adjusted version of the map in Theorem 3.1, we get that $G'$ does not always map into the desired set. For example, $312456 \in M_{4,3,3}$ would map to $314256$, which contains $34256 \sim 23145$ as a pattern, which is forbidden in $M_{4,3,2}$. However, the inverse map does always map into the desired set, which we can prove.

\begin{theorem}

Let $k \in \mathbb{Z}^{+}$, $2 \leq j \leq k$. Then $\av(M_{k,j,j}) \geq \av(M_{k,j,j\pm 1})$.

\end{theorem}

\begin{proof}

As alluded to in the paragraph above, we prove this by constructing an injection $\av(M_{k,j,j-1}) \hookrightarrow \av(M_{k,j,j})$. Let $p \in \av(M_{k,j,j-1})$. Let $A$ be the set of all elements of $p$ that can act as a $j$ in a $12\dots k$ pattern. For $a \in A$, let $h(a)$ be the leftmost element of $p$ that can act as a $j-1$ with $j$ in a $12\dots k$ pattern. Then $H$ reverses the subsequence $[h(a),a)$ for each $a \in A$.

By our almost distant pattern condition, we know that the subsequence $[h(a),a)$ is decreasing, and every element can act as a $j-1$ with $a$ in a $12\dots k$ pattern. Thus, $H(p) \in M_{k,j,j}$. The proof of injectivity is the same as in Theorem 3.1, by observing that we can identify the reversed subsequences from $H(p)$.

For $M_{k,j,j+1}$, note by reverse complement this set is Wilf-Equivalent to $M_{k,k+2-j,k+1-j}$. Thus we can conclude via the above paragraphs, $M_{k,j,j+1} \leq M_{k,k+2-j,k+2-j} = M_{k,j,j}$.

\end{proof}






While we don't get a Wilf-Equivalence between $M_{k,j,j}$ and $M_{k,j,j\pm 1}$, we are able to get an unbalanced Wilf-Equivalence that looks similar.

\begin{theorem}

Let $k \geq 2$, $2 \leq j \leq k+1$. Then there exists a finite set of permutations $S$ such that $M_{k,j,j} \subseteq S$ and $\av(M_{k,j,j-1}) = \av(S)$.

\end{theorem}

Note that this also gives a finite set $S$ with the same conditions for $M_{k,j,j+1}$ (with $1 \leq j \leq k$) by reverse complement.

\begin{proof}

We already have an injection $H: \av(M_{k,j,j-1}) \hookrightarrow \av(M_{k,j,j})$. If we can define a set of permutations $S$ such that the image of $H$ is exactly $\av(S)$, then we have our desired result.

We prove the existence of this set by considering how $H$ acts on permutations. Recall that given $p \in \av(M_{k,j,j-1})$, our map reverses the adjacent decreasing subsequences of $p$ that can act as $j-1$ in a $12\dots k$ pattern, and all elements in a given subsequence can only act as a $j-1$ with a single element acting as a $j$.

Now let $p \in \av(M_{k,j,j})$. The map we would like to use would take all elements that can act as a $j-1$ and reverse them. These elements all form adjacent increasing subsequences, and each subsequence has a unique element that can act as a $j$ but not a $j-1$ that they can act with. However, as discussed above, this does not work unless $j=2$. Why not?

If the permutation $p$ contains a pattern in $M_{j-2,j-1,j-1} \oplus 12\dots (k+2-j)$, and the $(j-1)^{\text{st}}$ element of that pattern is of rank at least $(j-1)$, then the image under $H^{-1}$ will contain a forbidden pattern. So our set $S$ will be the minimal set of permutations fulfilling those conditions, together with $M_{k,j,j}$.

\begin{example}
Consider $M_{k,3,2}$. Then our set $S$ will be $M_{k,3,3}$, and the minimal set of permutations containing a pattern in $M_{1,2,2} \oplus 12\dots (k-1)$, where the 2nd element of the pattern acts as a 2 in a $12\dots k$ pattern. This minimal set is exactly $3124\dots (k+1)(k+2)$. Here $32 \sim 21 \in M_{1,2,2}$, and thus $324\dots(k+1)(k+2) \sim 213\dots k(k+1)$.
\end{example}

To show this set gives us a Wilf-Equivalence, we need to show that our map $H$ goes into this set, and that this set it a bijection. We already know it is an injection, so we just need to show it is surjective.

First, we prove that $H$ does indeed map into our desired set. Consider $p \in \av(M_{k,j,j-1})$, and $H(p)$. Assume $H(p)$ contains a pattern described above. Let $a$ be the $(j-1)^{\text{st}}$ element of that pattern. Since by definition $a$ can act as a $(j-1)$ in a $12\dots k$ pattern, and has an increasing subsequence of length $k+2-j$ following it, it must be able to act as a $(j-1)$ in a $12\dots k(k+1)$ pattern. Since this pattern was forbidden in $\av(M_{k,j,j-1})$ this means that at least the first element following $a$, defined as $b$, in the pattern must have been before $a$ in $p$, and thus acts as a $(j-1)$ in $p$.

Then consider the part of the pattern that is in $M_{j-2,j-1,j-1}$. This pattern must have existed in $p$. So then the first $j-2$ elements of that pattern, followed by $b$, then $a$, then the remainder of a length $k$ monotone pattern with $b$ was present in $p$. But this is a forbidden pattern, since $a$ is smaller than at least the $(j-2)$ in the first part of the pattern. Thus $H(p) \in \av(S)$.

Now consider $p \in \av(S)$. We show that the image under the inverse map is in $\av(M_{k,j,j-1})$. The inverse map would be to take all elements that can act as a $j-1$, and reverse the adjacent subsequences of them.

If our result is not in $\av(M_{k,j,j-1})$, then we have a forbidden pattern. Ignoring the $(j-1)^{\text{st}}$ element of the pattern, we get a pattern in $M_{j-2,j-1,j-1} \oplus (k+1-j)$. But then since the $(j-1)^{\text{st}}$ element must have been after the $j^{\text{th}}$ element in $p$, we would have a forbidden pattern in $p$, and so we have a contradiction.

\end{proof}

It should be noted that these finite sets $S$ seem to grow in complexity as $j$ grows large. For example, $M_{k,2,1} = M_{k,2,2}$ exactly, and we have the following pair of results for $j=3,4$:

\begin{corollary}

Let $k \geq 2$, then
\begin{align*}
	\av(M_{k,3,2}) = \av(M_{k,3,3},312\dots(k+1)(k+2)).
\end{align*}
\end{corollary}

\begin{corollary}

Let $k \geq 3$, then
\begin{align*}
	\av(M_{k,4,3}) = \av\Big(M_{k,4,4},(1423,45123,35124) \oplus (12\dots (k-2))\Big).
\end{align*}
\end{corollary}

Both statements are direct results of Theorem 3.3, with the additional patterns being as defined in the proof.

While both have upper bound on length of patterns at $k+j$, it is unclear if this continues past $j=4$. Further research would be needed to determine an exact formula for the makeup of $S$ for arbitrary $j$.

\section{Further Questions}

It should be obvious that $L(M_{k,j,j\pm1}) \in [(k-1)^{2},(k-1)^{2}+1]$, but as of now there is no clear conclusion on where exactly they lie, or if any general answer can be found at all. There are cases at both extremes, as it was proven in \cite{bona} $L(M_{k,k+1,k}) = (k-1)^{2}$, and a consequence of Theorem 3.1 is $L(M_{k,2,1}) = (k-1)^{2}+1$.

Perhaps the most reasonable assumption would be that as $j$ approaches $k$, $L(M_{k,j,j-1})$ approaches $(k-1)^{2}$, but a proof of that fact is not clear as of yet. Alternatively, it could be the case that the growth rate is $(k-1)^{2}$ for all $j \neq 2$.

Beyond that, there is the question of growth rates for $M_{k,j,i}$ where $i \neq j,j\pm 1$. Outside of the case where $j = k+1$ covered in the introduction, there are no known results. There are trivial lower and upper bounds of $(k-1)^{2}$ and $k^{2}$. The former is obtained by $\av(12\dots k) \subseteq \av(M_{k,j,i})$, and the latter by observing $12\dots (k+1)$ is in $M_{k,j,i}$ whenever $j \neq i$, and when $j=i$ we know the growth rate. We cannot place an upper bound of $(k-1)^{2}+1$ on all almost distant monotone patterns, as we know $L(M_{3,3,1}) = (1+\sqrt{\phi})^{2} \approx 5.16$ \cite{callan,vatter}. Any better bounds would be interesting.

Finally there is also the question of almost distant patterns for non-monotone underlying patterns. Numerical evidence suggests several interesting equivalences between $M_{k,j,j}$ and other sets. Of particular interest, it appears $\av(M_{k,j,j}) = \av(14\sqbox{4}32)$, despite the fact that $\av(14\square 32)$ is not Wilf-Equivalent to any other distant patterns. In addition, there seems to be a Wilf-Equivalences between 6 almost distant patterns for 1342 and 1423, unrelated to the ones covered here.

\pagebreak

\end{document}